\documentclass[10pt]{article}
\usepackage{graphicx}
\usepackage{amsmath,amssymb,amsthm,amsfonts}
\usepackage{mathrsfs}
\usepackage{amssymb}

\usepackage{caption}
\usepackage{placeins}
\usepackage{booktabs}
\usepackage{psfrag,epsfig}
\usepackage{epstopdf}
\usepackage{color}
\usepackage{subcaption}

\newtheorem{thm}{Theorem}[section]

\newtheorem{lem}[thm]{Lemma}

\theoremstyle{definition}
\newtheorem{defn}{Definition}[section]
\theoremstyle{remark}
\newtheorem{rem}{Remark}[section]

\numberwithin{equation}{section}

\DeclareMathSymbol{\C}{\mathalpha}{AMSb}{"43}

\newcommand{\lam}{\lambda}

\newcommand{\bsub}{\begin{subequations}}
\newcommand{\esub}{\end{subequations}$\!$}

\begin{document}

\title{Entropy production rate and time-reversibility \\for general jump diffusions on $\mathbb{R}^n$}

\author{ Qi Zhang\thanks{ Beijing institute of mathematical sciences and applications, Beijing, 101408 China. Email: \texttt{qzhang@bimsa.cn}}, Yubin Lu \thanks{Department of Applied Mathematics, Illinois Institute of Technology, Chicago, IL 60616 United States. \textit{Corresponding author}. Email: ylu117@illinoistech.edu}}


\smallbreak \maketitle

\begin{abstract}

This paper investigates the entropy production rate and time-reversibility for general jump diffusions (L\'{e}vy processes) on $\mathbb{R}^n$. We first formulate the entropy production rate and explore its associated thermodynamic relations for jump diffusions. Subsequently, we derive the entropy production rate using the relative entropy between the forward and time-reversed path measures for stationary jump diffusions via the Girsanov transform. Furthermore, we establish the equivalence among time-reversibility, zero entropy production rate, detailed balance condition, and the gradient structure for stationary jump diffusions.

\end{abstract}

\vskip 0.2truein


\vskip 0.2truein

Entropy is a fundamental concept in thermodynamics and information theory that measures the amount of disorder or uncertainty in a system. The entropy production rate describes how quickly entropy increases over time due to irreversible dynamics, reflecting the presence of dissipation or time-asymmetry in the system. Jump diffusion models are commonly used in stochastic modeling to describe systems that exhibit both continuous fluctuations and sudden, discontinuous changes, which cannot be captured by standard Brownian motion alone. In this paper, we extend the theory of entropy production rate from classical diffusion processes to general jump diffusions. We further establish the equivalence between time-reversibility, zero entropy production, and the detailed balance condition in this broader setting.
\section{Introduction}

The breaking of time-reversal symmetry is a fundamental feature for nonequilibrium process in thermodynamics.
The entropy production rate quantifies the degree of time-irreversibility, serving as a key concept for understanding systems far from thermal equilibrium. 
In general, without energy or material exchange, the process has time-reversibility and zero entropy production rate if it admits equilibrium state after relaxation time.
Moreover, Nicolis and Prigogine \cite{NP77} argue that when an open system is driven by sustained energy or material inputs from its environment, the process may reach a nonequilibrium steady state characterized by a positive entropy production rate.
In nonequilibrium thermodynamics, the entropy production rate can be represented as the product of the thermodynamics fluxes and forces, 
and it can be distinguished between two terms: the entropy change due to interactions with an external reservoir or environment, and the entropy generated within the system itself, see e.g. \cite{DM62, NP77}. 

Diffusion processes are essential stochastic models for various nonequilibrium phenomena, including polymer dynamics in fluids, asset price fluctuations in financial models, and the motion of biological motor molecules.
The formulation of entropy production and associated thermodynamics relations for diffusion processes are discussed by \cite{GQ10,VE10, TD12}. 
The equivalence between self-adjointness of the generator, the gradient structure of the drift, time-reversibility, and zero entropy production for a diffusion process is considered in detail in \cite{JQQ04,QQT02,MRV02}. Additionally, the entropy production rate of a stationary process can be derived from the relative entropy (Kullback-Leibler divergence) between the forward and time-reversed path measures, as demonstrated by Maes and Netočný \cite{MN03}, and Da Costa and Pavliotis \cite{DP23}. 
Lebowitz and Spohn \cite{LS99} established a path-space framework connecting entropy production to large deviation theory and fluctuation theorems.
Seifert \cite{S05} further investigated entropy production along single trajectory, establishing the associated integral fluctuation theorem. Similar principles apply to discrete-state Markov jump processes, which model systems such as molecular networks, chemical reactions, and transport on graphs (see e.g. \cite{EV10, MNW08}). Recently, Boffi and Vanden-Eijnden \cite{BV24} develop a deep learning framework via the entropy production rate to estimate the score in generative modeling.

In recent years, signatures of jump noise have been found ubiquitous in nature. As a result, jump diffusions have garnered significant attention for modeling phenomena such as DNA target search for binding sites \cite{SNG2011}, active transport within cells \cite{LVS2015}, and the price of risky assets \cite{DK08}.  

It is interesting to ask the question whether one can extend the results for entropy production rate and time-reversibility from classical diffusions to diffusions with nonlocal jumps. 
However, this extension is nontrivial due to the simultaneous presence of local diffusion driven by Brownian noises and nonlocal jump noises.
Recently, the time reversal of jump diffusions under the finite relative entropy condition is discussed by \cite{CL22}, and its application on score-based generative model are considered in \cite{YP23, GeneratorMatching}. The entropy production formula and fluctuation theorem for non-Gaussian stochastic dynamics with associated deep learning framework is studied by \cite{HLMZ25}.
Nevertheless, the literature on entropy production and time-reversibility for general jump diffusions (L\'{e}vy processes) on $\mathbb{R}^n$ remains limited. This paper aims to derive the entropy production rate formula for jump diffusions, and reveal the relationship between entropy production and the degree of time-irreversibility. It is a  generalization of some classical results for diffusion models and discrete Markov chains  \cite{QQT02, MN03}. 

We are interested in the jump diffusion $(X_t)_{t\geq 0}$ which is given by the following SDE on $\mathbb{R}^n$:
\begin{equation}\label{LeqX}
 dX_t =  b(X_{t})dt +  \sqrt{2\beta^{-1}} a(X_{t^-}) \circ dB_t  +  \int_{\mathbb{R}^n\setminus\{0\}} \sigma(X_{t^{-}},z) N(dt,dz),
\end{equation}
where $B_t$ is an $n$-dimensional Brownian motion on the probability space $(\Omega, \mathscr{F}, \mathbb{P})$ equipped with a filtration $(\mathscr{F}_t)_{t\geq 0}$, $\circ$ denotes the Stratonovich integral, the constant $\beta:=\theta^{-1}$ denote the inverse temperature, $b(\cdot)$ is a drift from $\mathbb{R}^n$ to $\mathbb{R}^n$, $a(\cdot)=(a_{ij}(\cdot))_{n\times n}:\mathbb{R}^n \rightarrow  \mathbb{R}^{n \times n}$ is the diffusion coefficient, $N(dtdz)$ be an independent $n$-dimensional Poisson random measure with Poisson intensity rate $\lambda>0$ and L\'{e}vy measure $\nu(dz)$ on $\mathbb{R}^n \setminus\{0\}$, and the function $\sigma:\mathbb{R}^n \times \mathbb{R}^n \rightarrow  \mathbb{R}^{n}$ is the coefficient of jump. We present some preliminary results for jump-diffusions $(X_t)_{t\geq 0}$, and list some necessary assumptions for the SDE \eqref{LeqX} in Section \ref{Sec2}.

In section \ref{Sec3}, we first derive the entropy production rate for the jump diffusion as
\begin{align}\label{ERP}
    e_p(t) = & \beta \int_{\mathbb{R}^n} (A^{-1}(x)b(x) - \beta^{-1} \nabla \log\rho(t,x))^{T}A(x)(A^{-1}(x)b(x) - \beta^{-1} \nabla \log\rho(t,x)) \rho(t,x)dx \nonumber \\
     &+\frac{1}{2} \int_{\mathbb{R}^n}\int_{\mathbb{R}^n} (\rho(t,x)k(x,y) - \rho(t,y)k(y,x))\log \left( \frac{\rho(t,x)k(x,y)}{\rho(t,y)k(y,x)} \right) dydx,
\end{align}
where the matrix $A(\cdot) = a^T(\cdot)a(\cdot)$ is a positive definite matrix or a zero matrix, $\rho(t,\cdot)$ is the density of $X_t$, and $k(x,y)$ denotes the jump rate from $x$ to $y$. {Indeed, the jump rate can be expressed in terms of the coefficients of the process \eqref{LeqX}, see the expression in \eqref{kernel}.}
Due to this formula, the entropy production rate $e_p(t)$ of the jump diffusion $(X_t)_{t \geq 0}$ is always non-negative.
Here, the entropy production rate $e_p(t)$ is the sum of the variation of entropy in the system, $dS(t)/dt$, and the heat dissipation, $-h_{d}(t)$. Based on this formula, we also discuss some thermodynamics relations, include the Clausius inequality and the free energy dissipation for jump diffusions.

We remark that another formula for the entropy production rate of jump diffusions is provided in \cite{HLMZ25}.
Unlike our formula, theirs retains the quadratic form of the entropy production rate for jump diffusions, similar to the case of Langevin equations, by rewriting the associated nonlocal Fokker-Planck equation into a local form of the continuity equation.
In contrast, in our formula \eqref{ERP}, the nonlocal term arises naturally from the pure jump component in the SDE \eqref{LeqX}, which makes its interpretation more straightforward. Our formula \eqref{ERP} is also consistent with the entropy production rate for discrete jump processes in \cite{MN03}.

The another goal of this paper is to investigate the relationship between the above entropy production rate  and time-reversibility of the stationary jump diffusions.
One can use the stationary Markov process to model a physical system in a steady state (including the equilibrium state and nonequilibrium steady state).
In probability theory, a stationary Markov process with an invariant measure $\mu$ is a time-reversal process if for every $f,g\in L^2(\mathbb{R}^n, \mu)$ and $t,s >0$,
\begin{equation*}
    \mathbb{E}^{\mu}f(X_t)g(X_s) = \mathbb{E}^{\mu}f(X_t)g(X_s).
\end{equation*} 
In Section \ref{Sec4}, we consider the time reversal formula of the jump diffusion $(X_t)_{t\geq 0}$ in Lemma \ref{TRformula}.
Then using the Girsanov transform for the stationary jump diffusion, we deduce the entropy production as the relative entropy between the forward path measure and time-reversed path measure for the stationary jump diffusion:
\begin{equation*}
    e^{ss}_p = \frac{1}{T}\mathcal{H}(\mathcal{P}_{[0,T]} | \mathcal{P}^R_{[0,T]}) := \frac{1}{T}\mathbb{E}^{\mathcal{P}_{[0,T]}} \left[ \log \left(\frac{\mathcal{P}^R_{[0,T]}}{\mathcal{P}_{[0,T]}}\right) \right],
\end{equation*}
see Theorem \ref{eprpath} in detail. This result reveals the connection between the entropy production rate of a stationary Markov process and the time-reversal asymmetry of its path measure $\mathcal{P}_{[0,T]}$.

In Section \ref{Sec5}, we show the equivalence between time-reversibility, zero entropy production, detailed balance, and the gradient structure for stationary jump diffusions in Theorem \ref{TSeq}. {This theorem shows that for general jump diffusions, whenever the stationary measure is not reversible for the dynamics, then the entropy production rate in the steady state is strictly positive. Thus this result is a criteria to determine the reversibility of general jump diffusions, and it is a generalization of \cite{QQT02,MRV02}.} Using this theorem, we construct some examples of time-reversible/time-irreversible stationary jump diffusions in Section \ref{Sec6}.

The paper ends with some summary and discussions in Section \ref{Sec7}.

\section{Jump-diffusions and nonlocal Fokker-Planck equations}\label{Sec2}

In this section, we recall some basic definitions and useful results for jump-diffusions. These materials can be found in some classical references for jump-diffusions or L\'{e}vy processes, including \cite{A2004, S99, K19} and the references therein.

Let $(\Omega, \mathscr{F}, (\mathscr{F}_t)_{t\geq 0}, \mathbb{P})$ be a filtered probability space that satisfies the usual hypotheses.
Let $B_t$ be an $n$-dimensional Brownian motion, and $L_{t}$ be an independent $n$-dimensional compound Poisson process with associated Poisson intensity rate $\lambda>0$ and L\'{e}vy measure (jump size distribution) $\nu(dz)$ on $\mathbb{R}^n \setminus\{0\}$. Then the compound Poisson process $L_t := \sum_{i=1}^{n_t}D_i$, where $n_t$ is the counting variable of a Poisson process with rate $\lambda$, and $\{D_i\}_{i\geq 1}$ are i.i.d. random variables with distribution $\nu(dz)$, which are also independent of $n_t$. Then the Poisson random measure $N(dt,dy)$ associated to the compound Poisson process $L_t$ is defined as
\begin{equation*}
    N(t,B)(\omega):= \sharp \{s\in [0,t]:\Delta L_{s}(\omega) \in B\}, \quad t\geq 0, B\in \mathscr{B}(\mathbb{R}^{n}\backslash \{0\}),
\end{equation*}
so that $\mathbb{E}N(dt,dz) = \lambda dt\nu(dz) = \lambda m(z)dtdz$. Then the compound Poisson process $L_t$ is also given by
\begin{equation*}
    L_{t} = \int_{0}^t\int_{\mathbb{R}^n \setminus \{0\}} z N(dt,dz).
\end{equation*}
In this paper, we assume that the L\'{e}vy measure $\nu(dz)=m(z)dz$ satisfies the integrability condition
\begin{equation}\label{integrability}
    \int_{\mathbb{R}^n \setminus\{0\}} 1\wedge |z|^2 \nu(dz) = \int_{\mathbb{R}^n \setminus\{0\}} 1\wedge |z|^2 m(z)dz < \infty.
\end{equation}
The above integrability condition implies that the sample paths of $L_t$ have a finite quadratic variation with finitely many large jumps with an amplitude larger than 1, almost surely.

Using the compensated Poisson random measure $\widetilde{N}(dt,dy)=N(dt,dy)- \lambda dt\nu(dy)$, the SDE \eqref{LeqX} is equivalent with
\begin{align}
 dX_t = & \left[ b(X_{t}) + \beta^{-1}\nabla \cdot A(X_t) + \lambda\int_{0< |z| \leq 1} \sigma(X_{t^-},z) \nu(dz) \right]dt + \sqrt{2\beta^{-1}} a(X_{t^-})dB_t \nonumber \\
 & \quad + \int_{0<|z| \leq 1} \sigma(X_{t^{-}},z)\widetilde{N}(dt,dz) + \int_{|z|>1} \sigma(X_{t^{-}},z) N(dt,dz),
\end{align}
where $A(x) = a(x)a^T(x):= (A_{ij}(x))$.

Now we make some assumptions for the SDE \eqref{LeqX}. All these assumptions in this section are assumed to hold in the sequel unless otherwise specified.
\\
{\bf (A) Lipschitz condition} For all $k \in \mathbb{N}$ and $ x_{1}, x_{2} \in  \mathbb{R}^n$ with $|x_1| \vee |x_2| \leq k$, there exists a constant  $c_1 >0$ such that
\begin{equation*}
     |b(x_1) - b(x_2)|^2+\|A(x_1)- A(x_2)\| + \int_{0 <|z| <1} |\sigma(x_1,z) - \sigma(x_2,z)|^2 \nu (dz) \leq c_1 |x_1-x_2|^{2}.
\end{equation*}
Here and below, $\|\cdot\|$ denotes the Hilbert-Schmidt norm of a matrix, and $|\cdot|$ denotes the Euclidean norm.\\
{\bf (B) Growth condition} For all $k \in \mathbb{N}$ and $ x\in  \mathbb{R}^n$, there exists a constant  $c_2 >0$ such that
\begin{equation*}
 |b(x)|^2 + \|A(x)\| + \int_{0 <|z| <1} |\sigma(x,z)|^2 \nu (dz) < c_2(1 + |x|^2).
\end{equation*}

Under above assumption, we have the following well-posedness result of the SDE \eqref{LeqX} (see e.g. \cite[Theorem 6.2.9]{A2004} or \cite[Theorem 3.3.1]{K19}).

\begin{thm}\label{ExD}
Suppose that $({\bf A})$ and $({\bf B})$ hold, and impose the standard initial condition. Then the SDE \eqref{LeqX} admits a unique global strong solution $(X_t)_{t\geq 0}$.
\end{thm}

The solution $(X_t)_{t\geq 0}$ has the Markov semigroup $(P_t)_{t\geq 0}$, which is given by
\begin{equation*}
    P_t f(x)=\mathbb{E}(f(X_t)|X_0 =x), \quad \forall f \in L^1(\mathbb{R}^n, \mu).
\end{equation*}
We also denote the generator of the Markov process $(X_t)_{t\geq 0}$ by $\mathcal{L}$, i.e.
\begin{equation*}
    \mathcal{L}f = \lim_{t \rightarrow 0} \frac{1}{t}(P_t f - f),\quad \forall f \in L^1(\mathbb{R}^n, \mu).
\end{equation*}
Then the jump diffusion $X_t$ has generator
\begin{align*}
    \mathcal{L}f(x) = & b(x) \cdot\nabla f(x) + \beta^{-1} \nabla\cdot (A(x)\nabla f(x)) + \lambda\int_{\mathbb{R}^n \setminus \{ 0 \} } \left[ f(x+ \sigma(x,z)) - f(x)\right] d\nu(z).
\end{align*}
{\bf (C) Boundedness condition} the Jacobian matrix $\nabla_z \sigma(x,z)$ is non-degenerate for all $x,y \in \mathbb{R}^n$, and there exists a constant $c_2>0$ such that $\|\nabla_z \sigma(x,z)\| < c_2$ for all $x,y \in \mathbb{R}^n$.

Under the above boundedness condition, for every $x \in \mathbb{R}^n$, the map $z \mapsto \sigma(x,z)$ is a $C^1$-diffeomorphism, and admits an inverse denoted by $\sigma^{-1}(x,z)$. By the change of variables, the generator can also be rewritten as
\begin{align*}
    \mathcal{L}f(x) = & b(x) \cdot\nabla f(x) + \beta^{-1} \nabla\cdot (A(x)\nabla f(x))  \\
      & + \lambda \int_{\mathbb{R}^{n}\setminus \{0\}} [f(x+z)- f(x)] \det(\nabla_{z}\sigma^{-1}(x,z))m(\sigma^{-1}(x,z))dz \\
    = & b(x) \cdot\nabla f(x) + \beta^{-1} \nabla\cdot (A(x)\nabla f(x)) +   \int_{\mathbb{R}^{n}\setminus \{x\}} [f(y)-f(x)] k(x,y)dy,
\end{align*}
where 
{
\begin{equation}\label{kernel}
    k(x,y) = \lambda m(\sigma^{-1}(x,y- x))\det(\nabla_{y-x}\sigma^{-1}(x,y-x))
\end{equation}
 }
is the jump kernel of $X_t$. Note that $k(x,y)dtdy$ is a measure on $\mathbb{R}^{+}\times(\mathbb{R}^n\setminus\{0\})$ associated with the pure jump part $\Delta X_t = X_t - X_{t^{-}}$ of $(X_t)_{t\geq0}$. Now we introduce the Poisson random measure 
\begin{align*}
    N_X(t,B)(\omega):= \sharp \{s\in [0,t]:\Delta X_{s}(\omega) \in B\}, \quad t\geq 0, B\in \mathscr{B}(\mathbb{R}^{n}\backslash \{0\})
\end{align*}
so that $N_X(dt,dz):= \det(\nabla_z \sigma^{-1}(X_{t^{-}},z))N(dt, d\sigma^{-1}(X_{t^{-}},z))$. Then the associated compensated Poisson random measure 
\begin{align*}
    \widetilde{N}_X(dt,dz) = & N_X(dt,dz) - \lam\det(\nabla_z \sigma^{-1}(X_{t^{-}}, z))m(\sigma^{-1}(X_{t^{-}},z))dzdt \\
    = & N_X(dt,dz) -k(X_{t^{-}},X_{t^{-}} +z)dzdt
\end{align*}
generates a local martingales.
Using the Poisson random measure $N_X(dt,dz)$, we can rewrite
\begin{equation*}
    \int_{|z|>0} \sigma(X_{t^-},z) N(dt,dz) \overset{\text d}{=} \int_{|z|>0} z N_X(dt,dz),
\end{equation*}
and the SDE \eqref{LeqX} is also equivalent in distribution sense with
\begin{equation*}
     dX_t =  b(X_{t})dt +  \sqrt{2\beta^{-1}} a(X_{t^-}) \circ dB_t  + \int_{|z|>0} z N_X(dt,dz).
\end{equation*}

The existence of a smooth transition density of Markov processes with jumps has been investigated by Malliavin calculus, see e.g. Picard \cite{P96}, and Kunita \cite[Chapter 6]{K19}.
Moreover, for each distribution $\mu^{X_t}(dx)$ of $X_t$, its probability density function $\rho(x,t)$ satisfies the nonlocal Fokker-Planck equation
\begin{equation}\label{nonlocalFk}
    \partial_t \rho(x,t) = \mathcal{L}^{\ast}\rho(x,t),
\end{equation}
where $\mathcal{L}^{\ast}$ is the adjoint operator of the generator $\mathcal{L}$, which is given by
\begin{equation*}
    (\mathcal{L}^{\ast}\rho)(x) = \nabla\cdot [-b(x)\rho(x) + \beta^{-1} A(x)\nabla \rho(x)] +\int_{\mathbb{R}^n \setminus \{ x \} }k(y,x)\rho(y) - k(x,y)\rho(x)dy.
\end{equation*}

Similar with the local Fokker-Planck equation, we define the local probability current
\begin{equation}
    j^{loc}(t,x) := b(x)\rho(x) - \beta^{-1} A(x)\nabla \rho(x).
\end{equation}
Since there is a nonlocal jump part in the operator $\mathcal{L}^{\ast}$, for every $x \neq y$, we also define the nonlocal probability current between $x$ and $y$ as
\begin{equation}
    j^{nl}(t,x,y) = \rho(t,x)k(x,y) - \rho(t,y)k(y,x).
\end{equation}
By the local probability current $j^{loc}(t,x)$ and the nonlocal probability current $j^{nl}(t,x,y)$, we can rewrite the Fokker-Planck equation \eqref{nonlocalFk} as
\begin{equation*}
    \partial_t \rho(x,t) + \nabla\cdot j^{loc}(t,x) + \int_{\mathbb{R}^n \setminus \{x\}} j^{nl}(t,x,y) dy=0.
\end{equation*}
By the local probability current $j^{loc}(t,x)$ and the nonlocal probability current $j^{nl}(t,x,y)$, we introduce the definition of detailed balance for the jump diffusion.

\begin{defn}[Detailed balance]
The jump diffusions $(X_t)_{t\geq 0}$ satisfy the detailed balance condition if for any $x,y \in \mathbb{R}^n$ and $t>0$, its local probability current $j^{loc}(t,x)$ and the nonlocal probability current $j^{nl}(t,x,y)$ satisfy
\begin{equation*}
    j^{loc}(t,x)=0, \text{ and }  j^{nl}(t,x,y) = 0.
\end{equation*}
\end{defn}

In thermodynamics theory, the process is at equilibrium state if the process satisfies the detailed balance condition.

A probability measure $\mu$ is said to be an invariant measure of Markov semigroup $P_t$ if for every $f \in \mathcal{B}_b(\mathscr{B}(\mathbb{R}^n))$,
\begin{equation*}
    \int_{\mathbb{R}^n} P_t f(x) \mu(dx) = \int_{\mathbb{R}^n} f(x) \mu(dx).
\end{equation*}
Furthermore, if the invariant measure $\mu$ has a $C^2$ density $\rho_{ss}$, then $\rho_{ss}$ satisfies the stationary Fokker-Planck equation $\mathcal{L}^{\ast}\rho_{ss} = 0$.

In order to show the existence of invariant measure for SDE \eqref{LeqX}, we need the following Lyapunov condition of the generator $\mathcal{L}$.\\
{\bf (D) Lyapunov condition.} There exists a positive function $w_1 \in C^2(\mathbb{R}^n)$, and a positive compact function $w_2 \geq c_6|x|^p$ for some constant $c_6>0$ and $p \geq 1$, so that $\mathcal{L}w_1 \leq c_5 - w_2$ holds for some constant $c_5>0$.\\
{We remark that the Lyapunov condition is satisfied if
\begin{equation*}
    \lim_{x\to \infty} [b(x) \cdot x + \| A(x) \| + \int_{z \neq 0 } (|\sigma(x,z)|^2 \wedge 1) dz] < 0.
\end{equation*}
Thus as in Example 1 and Example 2 in Section 6, we can choose the drift $b$ to satisfy some dissipative condition to ensure that the Lyapunov condition holds.}
Under Lyapunov condition with $w_2$, the solution has the following uniform moment estimate $$\sup_{t\geq0}\mathbb{E}|X_t|^p < c^{-1}_{6} \sup_{t\geq0}\mathbb{E}w_2(X_t) < \infty.$$ 
The associated semigroup of SDE \eqref{LeqX} has an invariant probability measure $\mu$, see e.g. \cite{H80}. In physics, the process $(X_t)_{t\geq 0}$ is stationary with its invariant measure $\mu(dx)$ means that the corresponding dynamic is in a steady state (including the equilibrium state and nonequilibrium steady state).

We remark that under some dissipative condition of drift $b$, the generator $\mathcal{L}$ satisfies the Lyapunov condition. In \cite{XZ20}, the Lipschitz condition and Growth condition can be relaxed to local H\"{o}lder continuous condition and local growth condition with some dissipative drift $b$. Some explicit invariant measures of jump diffusion are constructed in \cite{ARW00}. 

Finally, in order to define the entropy production rate in next section, we also need the following assumption for the jump diffusion.\\

{\bf (E) Regularity condition.} For every $t \geq 0$, we assume that $\nabla\log \rho(t,\cdot)$ is $\mu^{X_t}-$a.s. bounded, where $\mu^{X_t}$ is the distribution of $X_t$. We also suppose that the jump kernel $k(x,y)$ satisfies that for every $t>0$,
\begin{equation}\label{ratiolim}
    \bar{k}(x,y):= k(x,y)\log \frac{k(x,y)}{k(y,x)} \in L^{\infty}(\mu^{X_t}(dx)dy),
\end{equation}
where $L^{\infty}(\mu^{X_t}(dx)dy)$ is the $L^{\infty}$ on the measure space $(\mathbb{R}^n\times \mathbb{R}^n,\mathscr{B}(\mathbb{R}^n\times \mathbb{R}^n), \mu^{X_t}(dx)dy)$. We remark that the condition \eqref{ratiolim} is proposed to ensure that the entropy production rate presented in Section~\ref{Sec3} is well-defined, see Remark~\ref{rmk:integrability_1} below for more details.
\begin{rem}
    Regularity condition (E) assumes that $\nabla\log \rho(t,\cdot)$ is bounded for every $t \geq 0$. A natural question is which systems ensure this boundedness, or how restrictive this condition is. While this remains an open question in the general case, it is worth noting that $\nabla\log \rho(t,\cdot)$ is known as {\it score} function in the machine learning community, particularly in the generative model community. For some partial answers, we refer the interested reader to \cite{DiffModelBound} and the references therein.
\end{rem}



\section{Entropy production rate from thermodynamics relations}\label{Sec3}

In order to give the entropy production rate formula for jump diffusions, we first define some important concepts from thermodynamics theory for jump diffusions, including the Gibbs entropy, internal energy, work, and heat dissipation in this section.

The Gibbs entropy of the jump diffusion $(X_t)_{t \geq 0}$ is 
\begin{equation*}
    S(t) = -\int_{\mathbb{R}^n}\rho(t,x)\log \rho(t,x) dx,
\end{equation*}
where $\rho(t,\cdot)$ is the probability density of $X_t$.

By Helmholtz-Hodge decomposition, the vector field can be decomposed as $$A^{-1}(x)b(x) = -\nabla V(x) + f^{loc}(x),$$ where $-\nabla V$ is the internal force from the potential $V$, and the divergence-free vector field $f^{loc}$ is the external force from outside the system. 
This decomposition for diffusion processes was well studied by \cite{A04, Q13}.
Moreover, motivated by the case of the discrete jump processes in \cite{MNW08}, we also decompose the jump kernel as
\begin{equation*}
    \frac{k(x,y)}{k(y,x)} = e^{\beta(f^{nl}(x,y) + V(x) - V(y))},
\end{equation*}
where $V(x) - V(y)$ is the difference in potential between each $x$ and $y$, and $f^{nl}(x,y) = -f^{nl}(y,x)$ is the external driving to the nonlocal transitions between each $x$ and $y$.

Since the diffusion $X_t$ is moving in the potential $V(x)$, the total internal energy of $X_t$ is given by
\begin{equation*}
    U(t):= \int_{\mathbb{R}^n} V(x) \rho(t,x) dx = \mathbb{E}V(X_t).
\end{equation*}
Recall that the work $\delta W = \text{force} \times \delta x$, where $x$ is the position, $\delta$ is the variation or derivative in general senses.
Then based on the probability currents $j^{loc}(t,x) $, and $j^{nl}(t,x,y)$, the change of total work $dW(t)/dt$ done by the external force $f^{loc}(x)$ and external driving $f^{nl}(x,y)$ is given by:
\begin{align}
   \frac{dW(t)}{dt} = & \int_{\mathbb{R}^n} f^{loc}(x)\cdot j^{loc}(t,x)dx + {\frac{1}{2}}\int_{\mathbb{R}^n}\int_{\mathbb{R}^n \setminus \{x\}} j^{nl}(t,x,y)f^{nl}(x,y)dydx.
\end{align}
The change of internal energy is 
\begin{align}
    \frac{d U(t)}{dt} = & \int_{\mathbb{R}^n} V(x) \mathcal{L}^{\ast} \rho(t,x)dx \nonumber\\
    = & -\int_{\mathbb{R}^n} V(x) \nabla\cdot j^{loc}(t,x)dx - \int_{\mathbb{R}^n}  V(x)\int_{\mathbb{R}^n \setminus \{x\}} j^{nl}(t,x,y) dy dx \nonumber \\
             = & \int_{\mathbb{R}^n} \nabla V(x) \cdot j^{loc}(t,x)dx - \frac{1}{2}\int_{\mathbb{R}^n} \int_{\mathbb{R}^n \setminus \{x\}} j^{nl}(t,x,y)  [V(x) - V(y)] dydx.
\end{align}
By the first law of thermodynamics, the infinitesimal heat exchange {due to the internal energy flows from the system to the reservoir and work done by the external force} at time $t$ is given by $dQ(t) =  dU(t) - dW(t).$ Thus the heat dissipation $h_{d}(t)$ of $X_t$ is given by
\begin{align}\label{HDr}
    h_{d}(t) = & \beta\frac{dQ(t)}{dt}\nonumber\\
             = & \beta\frac{d U(t)}{dt} - \beta \frac{dW(t)}{dt} \nonumber\\
             = & \beta\int_{\mathbb{R}^n} (\nabla V(x) - f^{loc}(x)) \cdot j^{loc}(t,x)dx \nonumber\\
             & - \frac{\beta}{2}\int_{\mathbb{R}^n} \int_{\mathbb{R}^n \setminus \{x\}} j^{nl}(t,x,y)  [f^{nl}(x,y) + V(x) - V(y)] dydx \nonumber \\
             = & -\beta\int_{\mathbb{R}^n} b^T(x) A^{-1}(x)\cdot j^{loc}(t,x)dx - \frac{1}{2}\int_{\mathbb{R}^n}\int_{\mathbb{R}^n \setminus \{x\}} j^{nl}(t,x,y)\log \left( \frac{k(x,y)}{k(y,x)} \right) dydx .
\end{align}
On the other hand, the variation of entropy in the system is
\begin{align}\label{Dentropy}
   \frac{dS(t)}{dt}  = & -\frac{d}{dt} \int_{\mathbb{R}^n}\rho(t,x)\log \rho(t,x) dx \nonumber \\
                  = & -\int_{\mathbb{R}^n} \partial_t \rho(t,x) (\log \rho(t,x) +1) dx \nonumber\\
                  = & - \int_{\mathbb{R}^n} \mathcal{L}^{\ast}\rho(t,x) (\log \rho(t,x) +1) dx \nonumber\\
                  = & -\int_{\mathbb{R}^n} j^{loc}(t,x)\nabla \log \rho(t,x)dx  +\frac{1}{2} \int_{\mathbb{R}^n}\int_{\mathbb{R}^n \setminus \{x\}} j^{nl}(t,x,y)\log \left( \frac{\rho(t,x)}{\rho(t,y)} \right) dydx \nonumber\\
                   = &  \beta\int_{\mathbb{R}^n}  \frac{(j^{loc}(t,x))^TA^{-1}(x)j^{loc}(t,x)}{\rho(t,x)} dx - \beta \int_{\mathbb{R}^n} b(x)^T A^{-1}(x) j^{loc}(t,x) dx \nonumber \\
                   & +\frac{1}{2} \int_{\mathbb{R}^n}\int_{\mathbb{R}^n \setminus \{x\}}j^{nl}(t,x,y)\log \left( \frac{\rho(t,x)k(x,y)}{\rho(t,y)k(y,x)} \right) dydx \nonumber\\
                   & -\frac{1}{2} \int_{\mathbb{R}^n}\int_{\mathbb{R}^n \setminus \{x\}} j^{nl}(t,x,y)\log \left( \frac{k(x,y)}{k(y,x)} \right) dydx.
\end{align}

\begin{rem}\label{rmk:integrability_1}
    The assumption ({\bf (E)}) ensures that the integrals in the last two line of \eqref{Dentropy} are well-defined, since that the condition~\eqref{ratiolim} for kernel $k(x,y)$ and the nonlocal flux $j^{nl}$ decay as $x$ and $y$ tend to infinity.
\end{rem}
\begin{rem}
    We point out that the energy $U$, heat $Q$, work $W$, and entropy production as defined in this paper are in fact the average energy, average heat, average work, and average entropy production. That is different to the stochastic work and heat along a single trajectory, see for example \cite{stochasticQuantity}.
\end{rem}

The total variation of entropy $e_p(t)$ can be written as the sum of the variation of entropy in the system, $dS(t)/dt$, and the variation of entropy in the reservoir, $-h_{d}(t)$. Combining \eqref{HDr} with and \eqref{Dentropy}, now we define entropy production rate formula for the jump diffusion as follows.

\begin{defn}
Under the assumptions in Section 2, the entropy production rate of the jump diffusion $(X_t)_{t\geq 0}$ which is given by the SDE \eqref{LeqX} is 
\begin{align}\label{epr}
    e_p(t) = & \frac{d S(t)}{dt} -  h_{d}(t) \nonumber\\
     = & \beta\int_{\mathbb{R}^n}  \frac{(j^{loc}(t,x))^TA^{-1}(x)j^{loc}(t,x)}{\rho(t,x)} dx +\frac{1}{2} \int_{\mathbb{R}^n}\int_{\mathbb{R}^n \setminus \{x\}} j^{nl}(t,x,y)\log \left( \frac{\rho(tx)k(x,y)}{\rho(t,y)k(y,x)} \right) dydx \nonumber\\
     = & \beta\int_{\mathbb{R}^n} (A^{-1}(x)b(x) - \beta^{-1} \nabla \log\rho(t,x))^{T}A(x)(A^{-1}(x)b(x) - \beta^{-1} \nabla \log\rho(t,x)) \rho(t,x)dx \nonumber \\
     &+\frac{1}{2} \int_{\mathbb{R}^n}\int_{\mathbb{R}^n \setminus \{x\}} (\rho(t,x)k(x,y) - \rho(t,y)k(y,x))\log \left( \frac{\rho(t,x)k(x,y)}{\rho(t,y)k(y,x)} \right) dydx.
\end{align}
\end{defn}

By the nonnegative property of entropy production rate, we immediately obtain the Clausius inequality for jump diffusion. This result implies the famous second law of thermodynamics for the jump diffusions. 

\begin{thm}[Clausius inequality]\label{SecondL}
Under the assumptions in Section 2, the entropy production rate of jump diffusion $X_t$ with temperature $\theta := \beta^{-1}$ satisfies the Clausius inequality
\begin{equation*}
   \theta\frac{dS(t)}{dt} -\frac{dQ(t)}{dt} = \theta e_{p}(t) \geq 0.
\end{equation*}
Moreover, the jump process is at the equilibrium state (detail balance) with the invariant measure $\mu(dx) = \rho_{ss}(x)dx$ if and only if $\theta\frac{dS(t)}{dt} - \frac{dQ(t)}{dt} = \theta e_p(t) = 0$.
\end{thm}
\begin{proof}
Note that $j^{nl}(t,x,y)= \rho(t,x)k(x,y) - \rho(t,y)k(y,x)$ and $\log \left( \frac{\rho(t,x)k(x,y)}{\rho(t,y)k(y,x)} \right)$ are both positive or negative. Thus
\begin{equation*}
    \int_{\mathbb{R}^n}\int_{\mathbb{R}^n \setminus \{x\}} (\rho(t,x)k(x,y) - \rho(t,y)k(y,x))\log \left( \frac{\rho(t,x)k(x,y)}{\rho(t,y)k(y,x)} \right) dydx \geq 0,
\end{equation*}
and the equality holds if and only if $j^{nl}(t,x,y)= \rho(t,x)k(x,y) - \rho(t,y)k(y,x) = 0$ for every $x,y \in \mathbb{R}^n$ and every $t \geq 0$.

Note that
\begin{align*}
    & \int_{\mathbb{R}^n} (A^{-1}(x)b(x) - \beta^{-1} \nabla \log\rho(t,x))^{T}A(x)(A^{-1}(x)b(x) - \beta^{-1} \nabla \log\rho(t,x)) \rho(t,x)dx \nonumber \\
    = & \int_{\mathbb{R}^n} \frac{(j^{loc}(t,x))^TA^{-1}(x)j^{loc}(t,x) }{\rho(t,x)}dx \geq 0,
\end{align*}
and the equality holds if and only if $j^{loc}(t,x) = 0$ for every $x\in \mathbb{R}^n$ and every $t \geq 0$.
Thus for every $t \geq 0$, we have
\begin{equation*}
    \theta \frac{dS(t)}{dt} - \frac{dQ(t)}{dt} = \theta \frac{dS(t)}{dt} - \theta h_{d}(t)= \theta e_{p}(t)\geq 0.
\end{equation*}
Moreover, for every $x,y\in \mathbb{R}^n$, $j^{loc}(t,x) = 0$ and $j^{nl}(t,x,y)=0$ almost surely, if and only if $e_{p}(t) =0$. It implies that process $X_t$ satisfies the detail balance condition at its invariant measure if and only if $\theta\frac{dS(t)}{dt} - dQ(t) =0$.
\end{proof}

\begin{rem}
    The thermodynamical properties of diffusion and jump processes have also been discussed in some works. We refer the interested reader to \cite{MNW08,S05,JQQ04,stochasticThermo1,stochasticThermo2} and the references therein.
\end{rem}

In isolated systems, thermodynamics theory yields that the total internal energy of the system changes due to exchange heat with the heat bath, and the free energy decreases until it reaches to its minimum at equilibrium.
Recall that the free energy of $X_t$ is defined as 
\begin{equation*}
    F(t) := U(t) - \theta S(t).
\end{equation*}

\begin{rem}
{When the stationary measure $\mu_{ss}$ has the Gibbs measure $\rho_{ss}(x) = Z^{-1}e^{-V(x)}$.
The free energy of $X_t$ is equivariantly given the relative entropy between $\mu_{t}$ and $\mu_{ss}$:
\begin{equation*}
    F(t)=  \beta^{-1}\mathcal{H}(\mu_{t} | \mu_{ss}) = \beta^{-1}\int_{\mathbb{R}^n} \log \left( \frac{\rho(t,x)}{\rho_{ss}(x)} \right)\rho(t,x)dx.
\end{equation*}}
\end{rem}

Now we show the connection between the entropy production rate and the free energy dissipation of irreversible jump diffusions in isolated system. 

\begin{thm}
Suppose that the assumptions in Section 2 hold,  and there is no external force $f^{nl}$ from outside in the drift, and for every $x,y \in \mathbb{R}^n$, the external driving to the nonlocal transitions  $f^{nl}(x,y)=0$. We also assume that the stochastic dynamics $X_t$ satisfies the detailed balance condition at Gibbs measure $\mu(dx) = \frac{1}{Z}e^{-\beta V(x)}dx$.
Then the dissipation rate of free energy is given by the entropy production rate
\begin{equation*}
    \frac{d F(t)}{dt} = - \theta e_{p}(t).
\end{equation*}
\end{thm}
\begin{proof}
Since there is no external force and external driving to the nonlocal transitions in the isolated system, the change of heat $dQ = -dU$. Then the heat dissipation rate as 
\begin{equation*}
    h_{d}(t) = \frac{1}{\theta} \frac{dQ}{dt} = -\beta\frac{d U(t)}{dt}.
\end{equation*}
Then by Theorem \ref{TSeq}, drift $b(x) = -A(x)\nabla V(x)$, and the kernel
\begin{equation*}
    k(x,y) = e^{-\beta[V(y)-V(x)]/2} s(x,y).
\end{equation*}
Based on the probability currents $j^{loc}(t,x)$ and $j^{nl}(t,x,y)$, the heat dissipation $h_{d}(t)$ is given by
\begin{align}
    h_{d}(t) = &- \beta\frac{d U(t)}{dt} \nonumber\\
             = & -\beta\int_{\mathbb{R}^n} \nabla V(x) \cdot j^{loc}(t,x)dx + \frac{\beta}{2}\int_{\mathbb{R}^n} \int_{\mathbb{R}^n \setminus \{x\}} j^{nl}(t,x,y)  [V(x) - V(y)] dydx \nonumber \\
             = & -\beta\int_{\mathbb{R}^n} b^T(x)A^{-1}(x)\cdot j^{loc}(t,x)dx - \frac{1}{2}\int_{\mathbb{R}^n}\int_{\mathbb{R}^n \setminus \{x\}} j^{nl}(t,x,y)\log \left( \frac{k(x,y)}{k(y,x)} \right) dydx .
\end{align}
Combining with \eqref{Dentropy}, we have
\begin{equation*}
     \frac{d F(t)}{dt} = \theta\frac{dS(t)}{dt} - \theta h_{d}(t) = - \theta e_{p}(t) \leq  0.
\end{equation*}
\end{proof}

\section{Connection with time forward/reversed path measures}\label{Sec4}

In this section, we will show that the total entropy production of the jump diffusions over the interval $[0,T]$ is given by the relative entropy between its forward and time-reversed path measures. Moreover, the entropy production rate quantifies the degree of time-irreversibility in the jump diffusions. 

We first recall the path measure of jump diffusions.
Let $(X_t)_{t\geq 0}$ be the jump process with a Markov transition density $p_t(x,y)$, which is given by the SDE \eqref{LeqX}. The Kolmogorov extension theorem implies that there exists a unique path probability measure $\mathcal{P}_{[0,T]}$ on the canonical (path) space $D([ 0,T], \mathbb{R}^n)$ such that for every $T>0$, $0 = t_{0} \leq t_1 \leq \dots \leq t_m \leq T$, and $E_j \in \mathscr{B}(\mathbb{R}^n), 1\leq j \leq m$, we have
\begin{equation*}
    \mathcal{P}(X_{t_0} \in E_0, \dots, X_{t_m} \in E_m) = \int_{E_0}dx_0\dots\int_{E_m} \Pi^m_{j-1} p_{t_j-t_{j-1}}(x_j,x_{j-1})p_0(x_0)dx_0 \dots dx_m.
\end{equation*}
Under the path measure $\mathcal{P}$, we can consider the process $(X_t)_{t\geq 0}$ as a random variable on the Skorokhod (path) space $D([ 0,T], \mathbb{R}^n)$, { which is the space of sample paths of the right continuous functions with left limits in the state space. The space $D([ 0,T], \mathbb{R}^n)$ is embedded with the Skorokhod metric:
\begin{equation*}
d(\vartheta_1, \vartheta_2)=\inf_{\lambda \in \Lambda} \left\{ \sup_{t \in [0,\infty)}|\lambda(t) - t| + \sup_{t \in [0,\infty)}\left| \vartheta_{1}(\lambda(t))-\vartheta_{2}(t)\right| \right\}, \quad \vartheta_1, \vartheta_2 \in D([ 0,T], \mathbb{R}^n),
\end{equation*}
where $\Lambda$ denotes the set of all strictly increasing continuous bijections from $[0,\infty)$ to itself, see e.g. \cite{A2004}}.


In stochastic analysis for jump diffusion, one can apply the Girsanov transform for absolutely continuous changes between two paths on the path space $D([0,T];\mathbb{R}^n)$, see e.g. \cite[Theorem 33.2]{S99} and \cite[Theorem 2.4]{FK21}.
In order to introduce the Girsanov transform, we first define the exponential process 
\begin{align*}
    Z_t:= & \exp\bigg( \int_0^t \sqrt{2\beta^{-1}} h^T(s,X_{s^-})a(X_{s^-}) dB_s -  \beta^{-1}\int_0^t h^T(s) A(X_{s^-}) h(s,X_{s^-}) ds \nonumber \\
    & \quad  + \int_0^t \int_{\mathbb{R}^n\setminus\{0\}} \log K(s, z) N_X(ds dz) - \int_0^t \int_{\mathbb{R}^n\setminus\{0\}} (K(s,X_{s^-},z)-1) k(X_{s^-}, X_{s^-}+z)dz ds \bigg),
\end{align*}
where $h: [0,T]\times \Omega \rightarrow \mathbb{R}^n$ is a predictable process, $\gamma \times \Omega \rightarrow \mathbb{R}^{n \times n}$ is a $n\times n$ value matrix predictable process, and $K: [0,T] \times (\mathbb{R}^n \setminus\{0\}) \times \Omega \rightarrow \mathbb{R}^n$ is a predictable process.
We also assume that the predictable processes $f, A, K$ satisfy that
\begin{align}\label{ConD}
    &\int^T_0 |h^{T}(s,X_{s^-})A(X_{s^-}) h(s,X_{s^-})|ds \nonumber \\
    & \quad +  \int_{\mathbb{R}^n\setminus\{0\}} |K(s,X_{s^-},z) +1|(\log^{+} |K(s,X_{s^-},z)|+1) k(X_{s^-}, X_{s^-}+z)dz ds < \infty
\end{align}
so that exponential process $Z_t$ is a uniformly integrable $\mathcal{P}$-martingale.

Let $\mathcal{P}$ be the path measure of $X_t$ on the canonical space $D([0,T],\mathbb{R}^n)$. 
Define a new probability measure $\mathcal{Q}$ on $(\mathscr{F}_t)_{t \geq 0}$ via the Radon-Nikodym derivative:
\begin{equation*}
     \frac{d\mathcal{Q}}{d\mathcal{P}} \bigg|_{\mathscr{F}_t} = Z_t.
\end{equation*}
Now we recall the Girsanov transform for jump diffusions (see e.g. \cite[Theorem 33.2]{S99} or \cite[Theorem 2.4]{FK21}).

\begin{thm}[Girsanov transform]
Under the path measure $\mathcal{Q}$, the jump diffusion $X_t$ solves the martingale problem for the modified generator
\begin{align}
    \mathcal{L}_t^{\mathcal{Q}}f := b^{\mathcal{Q}}(t,x) \cdot\nabla f(x) + \beta^{-1} \nabla\cdot (A(x)\nabla f(x)) +   \int_{\mathbb{R}^{n}\setminus \{x\}} [f(y)-f(x)] k^{\mathcal{Q}}(t,x,y)dy,
\end{align}
where 
\begin{equation*}
    b^{\mathcal{Q}}(t,x):= b(x) + a(x) h(t,x),
\end{equation*}
and
\begin{equation*}
    k^{\mathcal{Q}}(t,x,y):= K(s,x,y -x)k(y,x).
\end{equation*}
Moreover, under path measure $\mathcal{Q}$ the jump diffusion $X_t$ satisfies the new SDE:
\begin{align}
 dX_t = & \left[ b_t^{Q}(X_{t^-}) + \beta^{-1}\nabla \cdot A(X_{t^-}) \right]dt + \sqrt{2\beta^{-1}} a(X_{t^-})dB^{\mathcal{Q}}_t + \int_{\mathbb{R}^n \setminus\{0\}} z N^{\mathcal{Q}}_X(dt,dz).
\end{align}
Here, $B^{\mathcal{Q}}_t$ is a $\mathcal{Q}$-Brownian motion, and the Poisson random measure $N^{\mathcal{Q}}_X(dt,dz)$ has L\'{e}vy measure $k^{\mathcal{Q}}(t,X_{t^{-}}, X_{t^{-}}+z)dtdz$ under $\mathcal{Q}$.
\end{thm}

In order to give the time reversal path measure of $(X_t)_{t\in [0,t]}$, we need to consider its time reversal jump diffusion of $(X_t)_{t\geq 0}$ and the associated time reversal SDE.

\begin{lem}[Time Reversal Formula]\label{TRformula}
Let $(X_t)_{0 \leq t \leq T}$  be a jump diffusion satisfying the SDE \eqref{LeqX} with smooth density $\rho$. Assume that the assumptions in Section 2 hold.
{We further assume that for every $x \in \mathbb{R}^n$ and $t \in [0,T]$, $\nabla\log\rho(t,x)$ is Lipschitz continuous and
\begin{equation*}
    \int_{\mathbb{R}^n \setminus\{x\}} 1 \wedge |y-x|^2 \frac{\rho(t,y)}{\rho(t,x)} k(y,x) dy < \infty.
\end{equation*}}
Then the associated time-reversed process \( \{ X_t^R := X_{T - t} \}_{t \in [0, T]} \)  satisfying the following SDE
\begin{align}\label{LeqX_TR}
 dX^R_t = & b^{R}_t(X^R_{t})dt + \sqrt{2\beta^{-1}} a(X_{t^-}) \circ dB_t  +  \int_{\mathbb{R}^n\setminus\{0\}} z N^R(dt,dz) ,
\end{align}
where the time reversed $b^{R}_t(x) = -b(x) + 2\beta^{-1}A(x)\nabla\log\rho(T-t,x)$, and the time reversed Poisson random measure $N^R(dt,dz)$ has Poisson intensity $\lam$ and the L\'{e}vy measure 
\begin{equation*}
    \nu^R(dz) :=k^{R}_t(x,x+z)dz = \frac{\rho(T-t,x+z)k(x+z,x)}{\rho(T-t,x)k(x,x+z)}k(x,x+z)dz.
\end{equation*}
\end{lem}
\begin{rem}\label{rmk:integrability}
    The condition 
    \begin{equation*}
    \int_{\mathbb{R}^n \setminus\{x\}} 1 \wedge |y-x|^2 \frac{\rho(t,y)}{\rho(t,x)} k(y,x) dy < \infty
\end{equation*}
in Lemma~\ref{TRformula} ensures the integrability of the L\'evy measure for the time-reversed process \( \{ X_t^R \}_{t \in [0, T]} \), see~\eqref{integrability}.
\end{rem}
\begin{rem}
     Regularity condition (E) in Section~\ref{Sec2} requires that $\nabla\log\rho(t,x)$ is bounded. However this is not sufficient for ensuring the well-posedness of the time reversal SDE~\eqref{LeqX_TR}. Therefore, we assume that the term $\nabla\log\rho(t,x)$ is Lipschitz continuous in this section. 
\end{rem}

\begin{proof}
Note that the time reversal process $\{ X_t^R := X_{T - t} \}_{t \in [0, T]} $ has the density $\rho^{R}(t,x) = \rho(T-t,x)$. Using the Fokker-Planck equation \eqref{nonlocalFk}, the time reversal density $\rho^{R}(t,x) = \rho(T-t,x)$ satisfies
\begin{align*}
    & \partial_t \rho^{R}(t,x) \\
    = & -\partial_t \rho(T-t,x) \\
    = & \nabla\cdot[ b(x)\rho(T-t,x) - \beta^{-1}A(x)\nabla\rho(T-t,x)] \\
    & - \int_{\mathbb{R}^n\setminus\{0\}} k(y,x)\rho(T-t,y) - k(x,y)\rho(T-t,x) dy \\
    = & \nabla\cdot[ b(x)\rho^R(t,x) - \beta^{-1}A(x)\nabla\rho^R(t,x)] + \int_{\mathbb{R}^n\setminus\{0\}} k^{R}_t(y,x)\rho^R(t,y) - k^{R}_t(x,y)\rho^R(t,x) dy \\
    := & (\mathcal{L}^{R})^{\ast}\rho(t,x),
\end{align*}
where the time reversed jump 
\begin{equation*}
    k^{R}_t(x, y) = \frac{\rho(T-t, y)}{\rho(T-t,x)}k(y,x).
\end{equation*}
Thus the time reversed process $\{ X_t^R := X_{T - t} \}_{t \in [0, T]} $ is a Markov process with generator 
\begin{equation*}
    \mathcal{L}^R f(x) = b^{R}(x) \cdot\nabla f(x) + \beta^{-1} \nabla\cdot (A(x)\nabla f(x)) +   \int_{\mathbb{R}^{n}\setminus \{0\}} [f(y)-f(x)] k^{R}_t(x,y)dy.
\end{equation*}
This ensures that the time-reversed process satisfies the SDE~\eqref{LeqX_TR}.
\end{proof}

{We denote $\mathcal{P}_{[0,T]}$ and $\mathcal{P}^R_{[0,T]}$ as the path measure of the forward process $(X_t)_{t \in [0,T]}$ on the path space $D([0,T], \mathbb{R}^n)$.}
In the next theorem, we show the relationship between the entropy production rate and the relative entropy rate between the forward path measure $\mathcal{P}_{[0,T]}$, and the time reversal path measure $\mathcal{P}^R_{[0,T]}$.

\begin{thm}\label{eprpath}
Suppose that the assumptions in Section 2 hold.
Let $(X_t)_{t \geq 0}$ be a stationary jump-diffusion process defined by the SDE \eqref{LeqX} with its invariant measure $\mu(dx):=\rho_{ss}(x)dx$, so that $X_t \overset{\text d}{=} \mu$ for all $t \in [0,T]$, and the density $\rho_{ss}$ satisfies that $\nabla\log\rho_{ss}$ is Lipschitz continuous, and
\begin{equation*}
    \int_{\mathbb{R}^n \setminus\{x\}} 1 \wedge |y-x|^2 \frac{\rho_{ss}(y)}{\rho_{ss}(x)} k(y,x) dy < \infty. 
\end{equation*}
Then the entropy production rate $e_{p}^{ss}$ of  $(X_t)_{t \geq 0}$ satisfies
\begin{equation*}
    e^{ss}_p=\frac{1}{T}\mathcal{H}(\mathcal{P}_{[0,T]} | \mathcal{P}^R_{[0,T]}) = -\frac{1}{T}\mathbb{E}^{\mathcal{P}_{[0,T]}} \left[ \log \left(\frac{\mathcal{P}^R_{[0,T]}}{\mathcal{P}_{[0,T]}}\right) \right].
\end{equation*}
\end{thm}
\begin{rem}
    Similar to Remark~\ref{rmk:integrability}, the condition 
    \begin{equation*}
    \int_{\mathbb{R}^n \setminus\{x\}} 1 \wedge |y-x|^2 \frac{\rho_{ss}(y)}{\rho_{ss}(x)} k(y,x) dy < \infty
\end{equation*}
in Theorem~\ref{eprpath} ensures the integrability of the L\'evy measure for the process $(X_t)_{t \geq 0}$.
\end{rem}

\begin{proof}
Under the assumptions for jump rate $k(x,y)$ and $\rho_{ss}$, the jump diffusion $(X_t)_{t \geq 0}$ has its time-reversed process satisfying the SDE~\eqref{LeqX_TR}.
Applying the Girsanov transform for absolutely continuous changes between path measures $\mathcal{P}_{[0,T]}$ and $\mathcal{P}^R_{[0,T]}$, we obtain
\begin{align}
    \log \left(\frac{\mathcal{P}^R_{[0,T]}}{\mathcal{P}_{[0,T]}}\right) = & \beta\int^T_0 (b(X_t) - b^{R}(X_t)) A^{-1}(X_t)dB_t \nonumber \\ 
    & -\beta\int^{T}_0(b(X_t)- b^{R}(X_t))^{T}A^{-1}(X_t)(b(X_t) - b^{R}(X_t))dt \nonumber \\
    & +  \int_0^T \int_{\mathbb{R}^n\setminus\{0\}} \log \left( \frac{k(X_t + z, X_t)\, \rho_{ss}(X_t + z)}{k(X_t, X_t+z)\, \rho_{ss}(X_t)} \right) N_X(dtdz) \nonumber \\
    & - \int_0^t \int_{\mathbb{R}^n\setminus\{0\}} \left(\frac{k(X_t + z, X_t)\, \rho_{ss}(X_t + z)}{k(X_t, X_t+z)\, \rho_{ss}(X_t)}-1 \right) k(X_{t^-}, X_{t^-}+z)dz dt.
\end{align}
Using the martingale property, the relative entropy
\begin{align*}
    & \mathcal{H}(\mathcal{P}_{[0,T]} | \mathcal{P}^R_{[0,T]}) \nonumber\\
    = & -\mathbb{E}^{\mathcal{P}_{[0,T]}} \left[ \log \left(\frac{\mathcal{P}^R_{[0,T]}}{\mathcal{P}_{[0,T]}}\right) \right] \nonumber \\
    =& \frac{\beta}{4}\int^{T}_0\int_{\mathbb{R}^n}(b(x)- b^{R}(x))^{T}A^{-1}(x)(b(x) - b^{R}(x))\rho_{ss}(x)dx dt  \\
    & + \int_0^T \int_{\mathbb{R}^n}\int_{\mathbb{R}^n\setminus\{0\}} \log \left( \frac{k(x + z, x)\, \rho_{ss}(x + z)}{k(x, x+z)\, \rho_{ss}(x)} \right) k(x,x+z)\rho_{ss}(x)dz dx dt\\
    & + \int_0^T \int_{\mathbb{R}^n} \int_{\mathbb{R}^n\setminus\{0\}} \left(\frac{k(x + z, x)\, \rho_{ss}(x + z)}{k(x,x+z)\, \rho_{ss}(x)}-1 \right) k(x,x+z) \rho_{ss}(x)dz dxdt  \\
    = & \beta\int_0^T\int_{\mathbb{R}^n} (b(x) - \beta^{-1} A(x)\nabla \log\rho_{ss}(x))^{T}A(x)(b(x) - \beta^{-1} A(x)\nabla \log\rho_{ss}(x)) \rho_{ss}(x)dxdt \\
    & +\frac{1}{2} \int_{\mathbb{R}^n}\int_{\mathbb{R}^n} [k(y, x)\rho_{ss}(x) - k(x,x) \rho_{ss}(x)]\log \left( \frac{\rho(t,x)k(x,y)}{\rho(t,y)k(y,x)} \right) dydx  \\
    & + \int_0^T \int_{\mathbb{R}^n}\int_{\mathbb{R}^n\setminus\{x\}} k(y, x)\rho_{ss}(x) - k(x,x) \rho_{ss}(x) dy dx dt.
\end{align*}
Since $\mathcal{L}^{\ast}\rho_{ss}=0$, the last term is $0$. Thus we conclude that
\begin{equation*}
    \mathcal{H}(\mathcal{P}_{[0,T]} | \mathcal{P}^R_{[0,T]}) = Te^{ss}_p,
\end{equation*}
where the entropy production rate $e^{ss}_{p}$ is consistent with \eqref{epr} with the stationary density $\rho_{ss}$.
\end{proof}

The Theorem \ref{eprpath} shows that if the stationary jump process is reversible (i.e., it satisfies detailed balance), the forward path measure is the same as its time-reversed path measure, i.e. $\mathcal{P}_{[0,T]} = \mathcal{P}^R_{[0,T]}$. In this case, the relative entropy is zero, and thus the entropy production rate is zero. This corresponds to a system in thermal equilibrium.
If the process is irreversible, $\mathcal{P}_{[0,T]} \neq \mathcal{P}^R_{[0,T]}$ leading to a positive relative entropy and thus a positive entropy production rate. The magnitude of entropy production rate quantifies how statistically distinguishable forward paths are from their time-reversed versions, indicating the strength of the irreversibility.

\section{Time-reversibility}\label{Sec5}

In this section, we discuss relations between time-reversibility and detailed balance for general jump diffusions. 
Let $(X_t)_{t\geq 0}$ be a stationary jump diffusion  satisfying the SDE \eqref{LeqX} with initial measure $\mu(dx)$. Since $(X_t)_{t\geq 0}$ is a time-homogeneous Markov process, we define the time-reversibility as follows.

\begin{defn}
The invariant measure $\mu$ is called a reversible (probability) measure of $(X_t)_{t\geq 0}$ if for every $f,g \in C_{b}(\mathbb{R}^n)$,
\begin{equation*}
    \mathbb{E}^{\mu}[f(X_t)g(X_0)]  = \mathbb{E}^{\mu}[f(X_0)g(X_t)].
\end{equation*}
\end{defn}

The connection between self-adjointness of the generator of an ergodic Markov process in $L^2(\mu)$ and the time-reversibility Markov process is established by Fukushima and Stroock \cite[Theorem 2.3]{FS86}.
This result shows that an invariant measure $\mu$ is reversible with respect to $(X_t)_{t\geq 0}$ if and only if the generator $\mathcal{L}$ is symmetric in $D(\mathcal{L}) \cap L^2(\mathbb{R}^n, \mu)$, i.e. for every $f,g \in D(\mathcal{L}) \cap L^2(\mathbb{R}^n, \mu)$,
\begin{equation*}
    \int_{\mathbb{R}^n} f(x) \mathcal{L}g(x) \mu(dx) = \int_{\mathbb{R}^n} g(x) \mathcal{L} f(x) \mu(dx).
\end{equation*}

Now we claim our second main result as follows.

\begin{thm}\label{TSeq}
Suppose that the assumptions in Section 2 hold.
Let $\mu(dx) = \rho_{ss}(x)dx = \frac{1}{Z} e^{-\beta V(x)}dx$ be the Gibbs measure associated with potential $V$.
Then the following are equivalent: \\
(i) The stationary process $(X_t)_{t\geq 0}$ with invariant density $\rho_{ss}$ is time-reversible.\\
(ii) The system $X_t$ with its invariant measure $\mu(dx)=\rho_{ss}(x)dx$ is detailed balance.\\
(iii) The drift $b(x) = -A(x)\nabla V(x)$, and the kernel
\begin{equation*}
    k(x,y) = e^{-\beta[V(y)-V(x)]/2} s(x,y),
\end{equation*}
where $s(x,y) = s(y,x)$ is a nonnegative measurable symmetric function from $\mathbb{R}^n \times \mathbb{R}^n$ to $\mathbb{R}$.\\
(iv) The entropy production rate vanishes at the stationary measure $\mu(dx) :=\rho_{ss}(x)dx$, i.e. $e^{ss}_{p} = 0$ .
\end{thm}

\begin{proof}
(i)$ \Rightarrow (ii)$. The reversibility of $X_{t}$ implies that for every $f,g \in D(\mathcal{L}) \cap L^2(\mathbb{R}^n, \mu)$,
\begin{align*}
    & \langle \mathcal{L}f,g\rangle_{L^2(\mathbb{R}^n, \mu)} -\langle \mathcal{L}g,f\rangle_{L^2(\mathbb{R}^n, \mu)}  \\
    = &  \langle \mathcal{L}_{1}f,g\rangle_{L^2(\mathbb{R}^n, \mu)}-\langle \mathcal{L}_{1}g,f\rangle_{L^2(\mathbb{R}^n, \mu)} +\langle \mathcal{L}_{2}f,g\rangle_{L^2(\mathbb{R}^n, \mu)}  - \langle \mathcal{L}_{2}g,f\rangle_{L^2(\mathbb{R}^n, \mu)} \\
    = &0.
\end{align*}
where $\mathcal{L}_{1}$ is the local part of the generator $\mathcal{L}$, and $\mathcal{L}_{2}$ is the nonlocal part of the generator $\mathcal{L}$.
Using integration by part, we have
\begin{align*}
    & \langle\mathcal{L}_{1}f,g\rangle_{L^2(\mathbb{R}^n, \mu)}-\langle \mathcal{L}_{1}g,f\rangle_{L^2(\mathbb{R}^n, \mu)} \\
    = & \int_{\mathbb{R}^n} ( b(x)\cdot \nabla f(x) + \beta^{-1} \nabla\cdot (A(x)\nabla f(x))) g(x) \rho_{ss}(x)dx \\
    & - \int_{\mathbb{R}^n} ( b(x)\cdot \nabla g(x) + \beta^{-1} \nabla\cdot (A(x)\nabla g(x))) f(x) \rho_{ss}(x)dx \\
    = & \int_{\mathbb{R}^n} \left\langle b(x) - \beta^{-1}\frac{A(x)\nabla\rho_{ss}(x)}{\rho_{ss}(x)}, \nabla f(x)g(x) - \nabla g(x) f(x) \right\rangle \rho_{ss}(x) dx \\
    = & \int_{\mathbb{R}^n} \langle b(x) - \beta^{-1}A(x)\nabla \log \rho_{ss}(x), \nabla f(x)g(x) - \nabla g(x) f(x)\rangle \rho_{ss}(x) dx.
\end{align*}
For nonlocal term, we have
\begin{align*}
    & \langle \mathcal{L}_{2}f,g\rangle_{L^2(\mathbb{R}^n, \mu)}- \langle \mathcal{L}_{2}g,f\rangle_{L^2(\mathbb{R}^n, \mu)}  \\
    = & \int_{\mathbb{R}^n} \int_{\mathbb{R}^{n}\setminus \{x\}} [f(y) - f(x)] k(x,y)dy g(x)\rho_{ss}(x)dx \\
    & \quad - \int_{\mathbb{R}^n} \int_{\mathbb{R}^{n}\setminus \{x\}} [g(y) - g(x)] k(x,y)dy f(x)\rho_{ss}(x)dx \\
    = & \int_{\mathbb{R}^n} \int_{\mathbb{R}^{n}\setminus \{x\}} [f(y)g(x) - g(y)f(x)] k(x,y)dy \rho_{ss}(x)dx \\
    = & \int_{\mathbb{R}^n} \int_{\mathbb{R}^{n}\setminus \{x\}} f(y)g(x)  [ k(x,y)\rho_{ss}(x) - k(y,x)\rho_{ss}(y)]dy dx.
\end{align*}
Now we show that for every $f,g \in D(\mathcal{L}) \cap L^2(\mathbb{R}^n, \mu)$,
\begin{equation}\label{EqL1L2}
    \langle \mathcal{L}_{1}g,f\rangle_{L^2(\mathbb{R}^n, \mu)} - \langle \mathcal{L}_{1}f,g\rangle_{L^2(\mathbb{R}^n, \mu)} = \langle \mathcal{L}_{2}g,f\rangle_{L^2(\mathbb{R}^n, \mu)} - \langle \mathcal{L}_{2}f,g\rangle_{L^2(\mathbb{R}^n, \mu)}  = 0
\end{equation}
If for every $f,g\in D(\mathcal{L}) \cap L^2(\mathbb{R}^n, \mu)$, 
\begin{equation*}
    \langle \mathcal{L}_{1}g,f\rangle_{L^2(\mathbb{R}^n, \mu)} - \langle \mathcal{L}_{1}f,g\rangle_{L^2(\mathbb{R}^n, \mu)} \neq 0, \quad \langle \mathcal{L}_{2}g,f\rangle_{L^2(\mathbb{R}^n, \mu)} - \langle \mathcal{L}_{2}f,g\rangle_{L^2(\mathbb{R}^n, \mu)} \neq 0
\end{equation*}
and 
\begin{equation}\label{EqdiffL12}
    \langle \mathcal{L}g,f\rangle_{L^2(\mathbb{R}^n, \mu)} - \langle \mathcal{L}f,g\rangle_{L^2(\mathbb{R}^n, \mu)} = 0
\end{equation}
then there exists a differential operator $\Bar{\mathcal{L}}_1$ and a nonlocal integral operator $\Bar{\mathcal{L}}_2$, so that for every $f,g\in D(\mathcal{L}) \cap L^2(\mathbb{R}^n, \mu)$,
\begin{equation*}
    \langle \mathcal{L}_{1}g,f\rangle_{L^2(\mathbb{R}^n, \mu)} = \langle \mathcal{L}^{\ast}_{1\mu}f,g\rangle_{L^2(\mathbb{R}^n, \mu)} = \langle (\mathcal{L}_{1} + \Bar{\mathcal{L}}_1)f,g\rangle_{L^2(\mathbb{R}^n, \mu)}
\end{equation*}
and 
\begin{equation*}
    \langle \mathcal{L}_{2}g,f\rangle_{L^2(\mathbb{R}^n, \mu)} = \langle \mathcal{L}^{\ast}_{2\mu}f,g\rangle_{L^2(\mathbb{R}^n, \mu)} = \langle (\mathcal{L}_{2} + \Bar{\mathcal{L}}_2)f,g\rangle_{L^2(\mathbb{R}^n, \mu)},
\end{equation*}
where $\mathcal{L}^{\ast}_{1\mu}$, $\mathcal{L}^{\ast}_{2\mu}$ are adjoint operators of $\mathcal{L}_1$, $\mathcal{L}_2$ with
respect to $L^2(\mathbb{R}^n, \mu)$.
Then by \eqref{EqdiffL12}, we have
\begin{equation*}
    \langle \Bar{\mathcal{L}}_1 f,g\rangle_{L^2(\mathbb{R}^n, \mu)} = \langle \Bar{\mathcal{L}}_2f,g\rangle_{L^2(\mathbb{R}^n, \mu)}
\end{equation*}
holds for every $f,g\in D(\mathcal{L}) \cap L^2(\mathbb{R}^n, \mu)$.
And it follows that $\Bar{\mathcal{L}}_1 = \Bar{\mathcal{L}}_2$.
However, it is impossible since $\Bar{\mathcal{L}}_1$ is a differential operator, and $\Bar{\mathcal{L}}_2$ is a nonlocal operator.
Thus we can exclude the case $\langle \mathcal{L}_{1}g,f\rangle_{L^2(\mathbb{R}^n, \mu)} - \langle \mathcal{L}_{1}f,g\rangle_{L^2(\mathbb{R}^n, \mu)} \neq 0$,  $\langle \mathcal{L}_{2}g,f\rangle_{L^2(\mathbb{R}^n, \mu)} - \langle \mathcal{L}_{2}f,g\rangle_{L^2(\mathbb{R}^n, \mu)} \neq 0$ and $\langle \mathcal{L}g,f\rangle_{L^2(\mathbb{R}^n, \mu)} - \langle \mathcal{L}f,g\rangle_{L^2(\mathbb{R}^n, \mu)} = 0$, and we obtain that for every $f,g \in D(\mathcal{L})$, \eqref{EqL1L2} holds. 

Furthermore, Since $f,g \in D(\mathcal{L})$ are arbitrary, we have for every $x,y \in \mathbb{R}^n$,
\begin{equation}
    j^{nl}(x,y)= \rho_{ss}(x)k(x,y) - \rho_{ss}(y)k(y,x) = 0.
\end{equation}
and
\begin{equation*}
    j^{loc}(t,x) := (b(x) - \beta^{-1} A(x)\nabla \log \rho_{ss}(x))\rho_{ss}(x) = 0.
\end{equation*}
This implies that the system $X_t$ is detailed balance at $\rho_{ss}$.\\
(ii) $\Rightarrow$ (iii). If the system $X_t$ is detailed balance at $\rho_{ss}(x) = Z^{-1} e^{-\beta V(x)}$, then 
\begin{equation*}
    j^{loc}(t,x) = (b(x) - \beta^{-1}A(x) \nabla \log\rho_{ss}(x))\rho_{ss}(x) = 0.
\end{equation*}
Thus the drift $b(x) = -A(x)\nabla V(x)$. 
Moreover, since 
\begin{equation*}
    j^{nl}(t,x,y) =e^{-\beta V(x)}k(x,y) - e^{-\beta V(y)}k(y,x) =0,
\end{equation*}
we have $k(x,y)/k(y,x) = e^{\beta(V(x) - V(y))}$.
We denote the nonnegative measurable symmetric function $s(x,y) = [k(x,y)k(y,x)]^{1/2}$.
Thus the kernel
\begin{equation*}
    k(x,y) = e^{-\beta[V(y)-V(x)]/2} s(x,y).
\end{equation*}
(iii) $\Rightarrow$ (i). If $\rho_{ss}(x) = Z^{-1} e^{-V(x)}$, $b = -A\nabla V$, and the kernel
\begin{equation*}
    k(x,y) = e^{-\beta[V(y)-V(x)]/2} s(x,y),
\end{equation*}
then for every $x,y \in \mathbb{R}^n$,
\begin{equation}
    \rho_{ss}(x)k(x,y) - \rho_{ss}(y)k(y,x) = 0
\end{equation}
and
\begin{equation*}
    b(x)\rho_{ss}(x) - \beta^{-1} A(x)\nabla \log \rho_{ss}(x) = 0.
\end{equation*}
For every $f,g \in D(\mathcal{L}) \cap L^2(\mathbb{R}^n, \mu)$, we can verify
\begin{align*}
    & \int_{\mathbb{R}^n} f(x) \mathcal{L}_{1}g(x) \rho_{ss}(x)dx \\
    = & \int_{\mathbb{R}^n} g(x) \mathcal{L}_{1}^{\ast}(f(x) \rho_{ss}(x))dx \\
    = & \int_{\mathbb{R}^n} g(x) \mathcal{L}_{1}f(x) \rho_{ss}(x) + g(x)\langle \nabla f(x), b(x)\rho_{ss}(x) - \beta^{-1} \nabla A(x) \rho_{ss}(x) \rangle dx \\
     = & \int_{\mathbb{R}^n} g(x) \mathcal{L}_{1}f(x) \rho_{ss}(x)dx
\end{align*}
and
\begin{align*}
    \int_{\mathbb{R}^n} f(x) \mathcal{L}_{2}g(x) \rho_{ss}(x)dx = & \int_{\mathbb{R}^n} g(x) \int_{\mathbb{R}^{n}\setminus \{x\}}f(y) \rho_{ss}(y) k(y,x) - f(x) \rho_{ss}(x) k(x,y)dy)dx \\
    = &  \int_{\mathbb{R}^n} g(x) \int_{\mathbb{R}^{n}\setminus \{x\}} [f(y) - f(x)] k(x,y)dy\rho_{ss}(x) dx  \\
     = & \int_{\mathbb{R}^n} g(x) \mathcal{L}_{2}f(x) \rho_{ss}(x)dx.
\end{align*}
Therefore, we have
\begin{equation*}
    \int_{\mathbb{R}^n} f(x) \mathcal{L}g(x) \rho_{ss}(x)dx = \int_{\mathbb{R}^n} g(x) \mathcal{L} f(x) \rho_{ss}(x)dx.
\end{equation*}
This implies that the stationary process $X_t$ with Gibbs measure $\mu$ is time-reversible.\\
(ii) $\Leftrightarrow$ (iv). The equivalent between detailed balance and zero entropy production rate is shown in Theorem \ref{SecondL}.
\end{proof}

\section{Examples}\label{Sec6}

Theorem \ref{TSeq} implies that the time irreversibility of jump processes is coming from the non-symmetrical property of the generator $\mathcal{L}$ with respect to the weighted $L^2$ space $L^2(\mathbb{R}^n,\mu)$. This breaking of symmetry is the source of the entropy production rate to the jump processes. For time reversal jump processes with the Gibbs measure $\mu(dx) = \frac{1}{Z}e^{-V(x)}dx$, some (nonlocal) gradient structures of the potential $V$ are appearing in the associated SDEs.
By Theorem \ref{TSeq}, we construct some examples of time-reversible/time-irreversible stationary jump process.

\noindent\textbf{Example 1.} Consider the SDE
\begin{equation}\label{EX1SDE}
    dX_t = -\nabla V(X_t) dt +\sqrt{2} dB_t + \int_{\mathbb{R}^n\setminus \{0\}} (z-X_t)N(dt,dz),
\end{equation}
where $B_t$ is an $n$-dimensional Brownian motion, and $N_t$ is the Poisson random measure with L\'{e}vy measure $dt\nu(dz):= e^{-V(z)}dz$.

Then the jump diffusion has the generator
\begin{align*}
    \mathcal{L}f(x) = & -\nabla V(x) \cdot\nabla f(x) + \Delta f(x) + \int_{\mathbb{R}^n \setminus \{ x \} } \left[ f(z) - f(x)\right] e^{-V(z)}dz \\
    = & -\nabla V(x)\cdot\nabla f(x) + \Delta f(x) \\
    & + \int_{\mathbb{R}^n \setminus \{ x \} } \left[ f(z) - f(x)\right] \frac{1}{\sqrt{2\pi}} e^{-(V(z) - V(x))/2} e^{-(V(z) + V(x))/2} dz,
\end{align*}
which satisfies the condition in Theorem \ref{TSeq}, (iii).
Moreover, the adjoint of the generator $\mathcal{L}$ is
\begin{equation*}
    \mathcal{L}^{\ast}\rho(x) = \nabla\cdot[\nabla V(x) \rho(x)] + \Delta\rho(x) + C \int_{\mathbb{R}\setminus\{x\}} \rho(y) e^{-V(x)} - \rho(x) e^{-V(y)} dy.
\end{equation*}
Thus, we can check that the Gibbs distribution $\mu(dx):= \frac{1}{Z} e^{-V(x)}dx$ is an invariant measure of the jump process $X_t$, and the generator $\mathcal{L}$ is symmetric in $D(\mathcal{L}) \cap L^2(\mathbb{R}^n, \mu)$. With the invariant measure $\mu(dx) = \frac{1}{Z}e^{-V(x)}dx$, the jump diffusion $X_t$ satisfies the detailed balance:
\begin{equation*}
    j^{loc}_{ss}(x)= -\nabla V(x) Z^{-1}e^{-V(x)} +  \nabla V(x) Z^{-1}e^{-V(x)}= 0
\end{equation*}
and
\begin{equation*}
    j^{nl}_{ss}(x,y) = \frac{1}{Z}e^{-V(x)}e^{- V(y)} - \frac{1}{Z}e^{-V(y)}e^{-V(x)} = 0.
\end{equation*}
Thus the above jump diffusion $X_t$ has zero entropy production rate when $X_t \overset{\text d}{=} \mu(dx)$.
By Theorem \ref{TSeq}, the invariant measure $\mu(dx):= \frac{1}{Z} e^{-V(x)}dx$ is also a reversible measure of jump diffusion $X_t$. 


Now we choose the potential $V(x) = |x|^2/2$, initial distribution $X_0 \overset{\text d}{=} \frac{1}{\sqrt{2 \pi }} e^{-|x-3|^2/2}  dx$, and dimension $n=1$ in the SDE \eqref{EX1SDE} as in our numerical example. The associated  equilibrium state is given by the one-dimensional standard Gaussian distribution: $\mu(dx) = \frac{1}{\sqrt{2 \pi }} e^{-|x|^2/2}  dx$.
Figure \ref{fig:Ex1} illustrates the evolution of the density and the entropy production rates for this example.  In this case, entropy production rates eventually decay to zero, and the system relaxes to equilibrium state.\\
\begin{figure}[htbp]
    \centering
    \begin{subfigure}[b]{0.47\textwidth}
        \centering
        \includegraphics[width=0.9\textwidth, trim=20 0 20 500, clip]{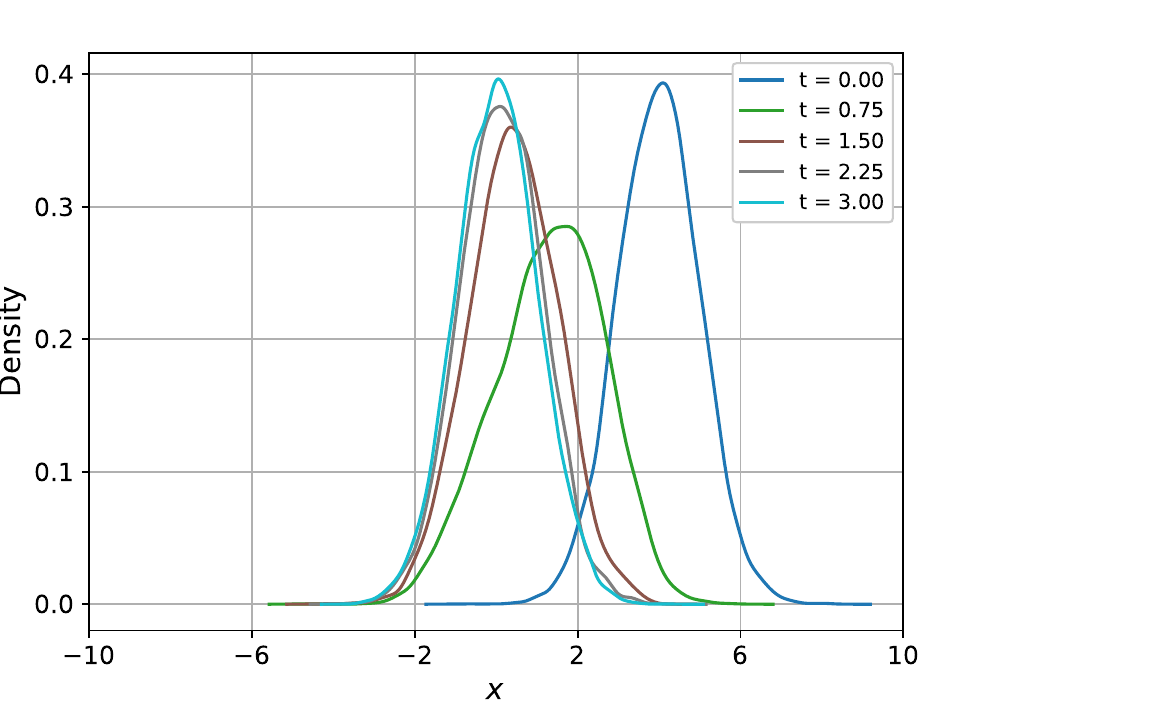}
        \caption{The evolution of the density.}
        \label{fig:Ex1_density}
    \end{subfigure}
    \hfill
    \begin{subfigure}[b]{0.47\textwidth}
        \centering
        \includegraphics[width=0.9\textwidth,trim=10 0 10 500, clip]{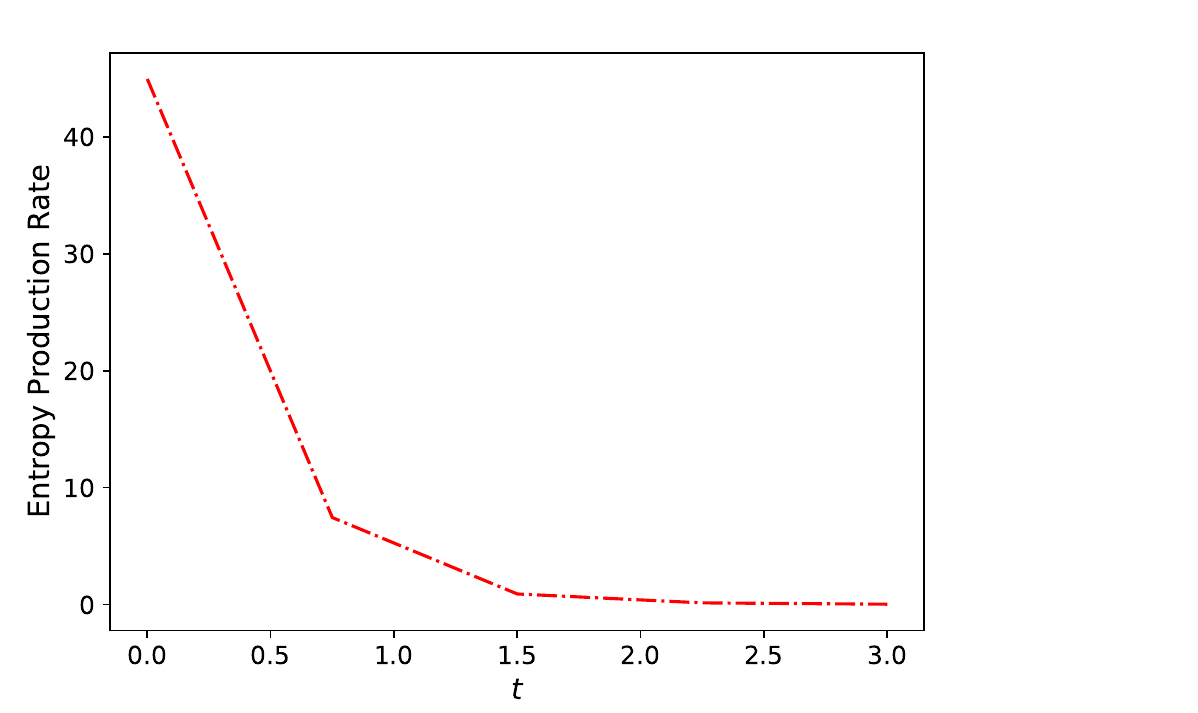}
        \caption{The evolution of the entropy production rate.}
        \label{fig:Ex1_EPR}
    \end{subfigure}
    \caption{The system in Example 1.}
    \label{fig:Ex1}
    \end{figure}

\noindent\textbf{Example 2}
Consider the following one-dimensional jump-diffusion process:
\begin{equation}\label{EX2SDE}
dX_t = -X_t\,dt + dL^{\alpha}_t,
\end{equation}
where $L^{\alpha}_t$ is a one-dimensional $\alpha$-stable Lévy process with Lévy process $\nu(dz) = C_{1,\alpha}|z|^{-(1+\alpha)}dz$. 
Then the above jump diffusion $X_t$ has generator
\begin{align*}
  \mathcal{L}f(x) = & -x f'(x) + \int_{\mathbb{R}\setminus\{0\}} [f(x+z) - f(x)]\frac{C_{n,\alpha}}{|z|^{1+\alpha}}dz \\
  = & -x f'(x) + \int_{\mathbb{R} \setminus \{x\}} [f(y) - f(x)] k(y,x)dy,
\end{align*}
where $k(x,y) = C_{n,\alpha}|x-y|^{-(1+\alpha)}$.\\
By \cite[Proposition 3.2]{ARW00}, this process is known to admit the invariant measure $\mu$ with density:
\begin{equation}
\rho_{ss}(x) = \frac{1}{2\pi} \int_{\mathbb{R}} e^{i\xi x} e^{-\frac{1}{\alpha}|\xi|^{\alpha}}d\xi = \frac{1}{\pi} \int^{\infty}_0 \cos (x\xi) e^{-\frac{1}{\alpha}|\xi|^{\alpha}}d\xi.
\end{equation}
Let the initial measure be the above invariant measure $\mu$, so that the solution of the SDE \eqref{EX2SDE} is a stationary jump process.
We can verify that generator $\mathcal{L}$ is not symmetric in $D(\mathcal{L}) \cap L^1(\mathbb{R}^n, \mu)$. Thus the jump process $X_t$ with above invariant measure is time irreversible. With the above invariant measure, the nonlocal jump rate
\begin{equation*}
    j^{nl}_{ss}(x,y) = \rho_{ss}(x)k(x, y) - \rho_{ss}(y)k(y,x) \neq 0, \quad \text{a.s.}.
\end{equation*}
Then the log ratio in the formula of entropy production rate becomes:
$$ \log \left( \frac{\rho_{ss}(x)k(x, y)}{\rho_{ss}(y)k(y,x)} \right) \neq 0. $$
Moreover, there is no diffusion term in the SDE \eqref{EX2SDE}, i.e. $A(\cdot) = 0$. Thus the entropy production rate with stationary density $\rho_{ss}$ is given by
\begin{equation}
e^{ss}_p = \frac{1}{2} \int_{\mathbb{R}} \int_{\mathbb{R}\setminus\{x\}} (\rho_{ss}(x)k(x, y) - \rho_{ss}(y)k(y,x)) \log \left( \frac{\rho_{ss}(x)k(x, y)}{\rho_{ss}(y)k(y,x)} \right)dydx.
\end{equation}
Figure~\ref{fig:Ex2_density} illustrates the stationary densities and the associated entropy production rates for this example with different index $\alpha$. Figure~\ref{fig:Ex2_EPR} implies that the entropy production rate for the SDE \eqref{EX2SDE} stabilizes to a positive steady-state value. 
It implies that the jump process by the SDE \eqref{EX2SDE} with above invariant measure is time irreversible, which is consistent with Theorem \ref{TSeq}.

\begin{figure}[htbp]
    \centering
    \begin{subfigure}[b]{0.47\textwidth}
        \centering
        \includegraphics[width=0.9\textwidth, trim=0 400 0 0, clip]{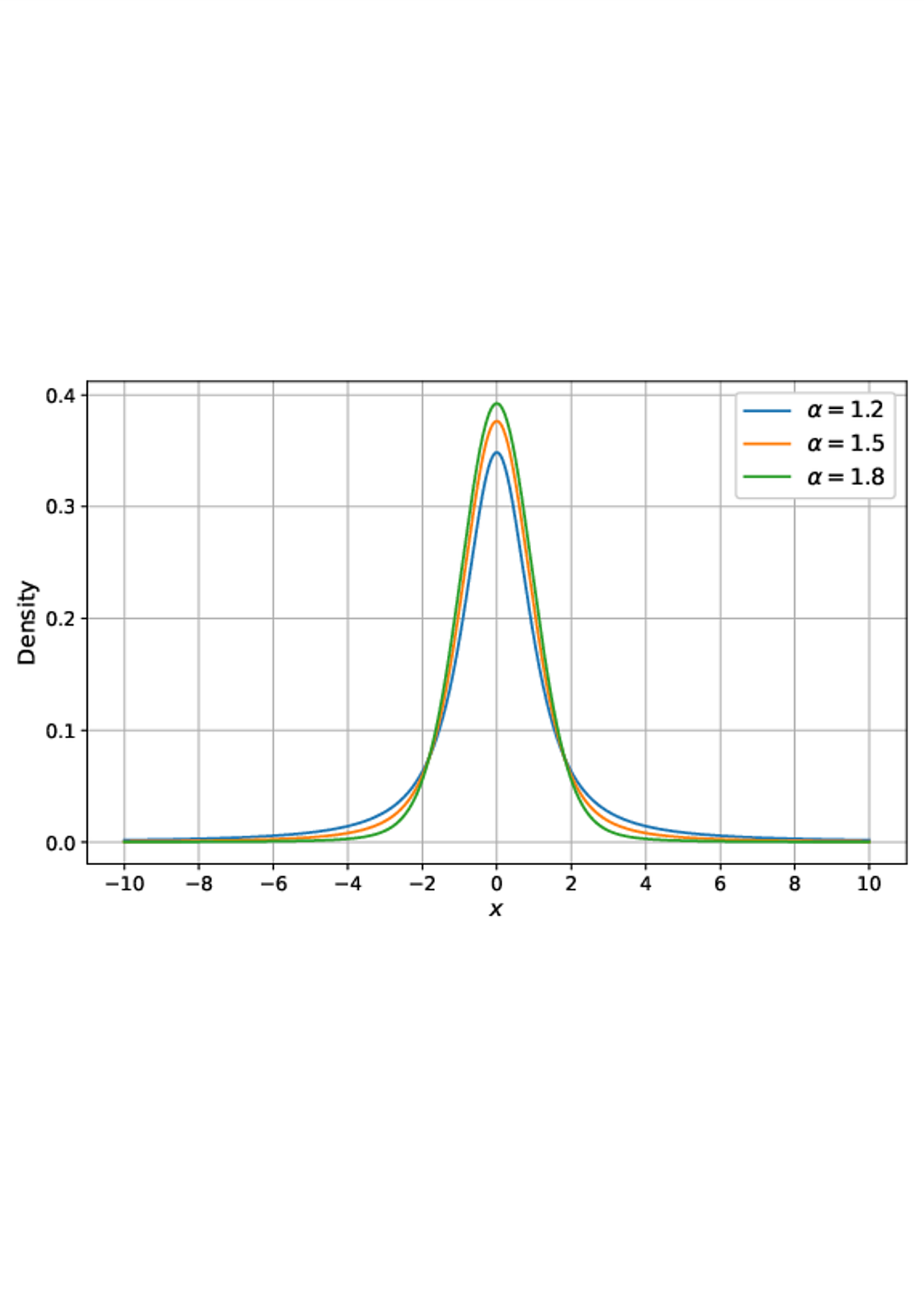}
        \caption{Stationary densities in Example 2.}
        \label{fig:Ex2_density}
    \end{subfigure}
    \hfill
    \begin{subfigure}[b]{0.47\textwidth}
        \centering
        \includegraphics[width=0.9\textwidth,trim=0 400 0 0, clip]{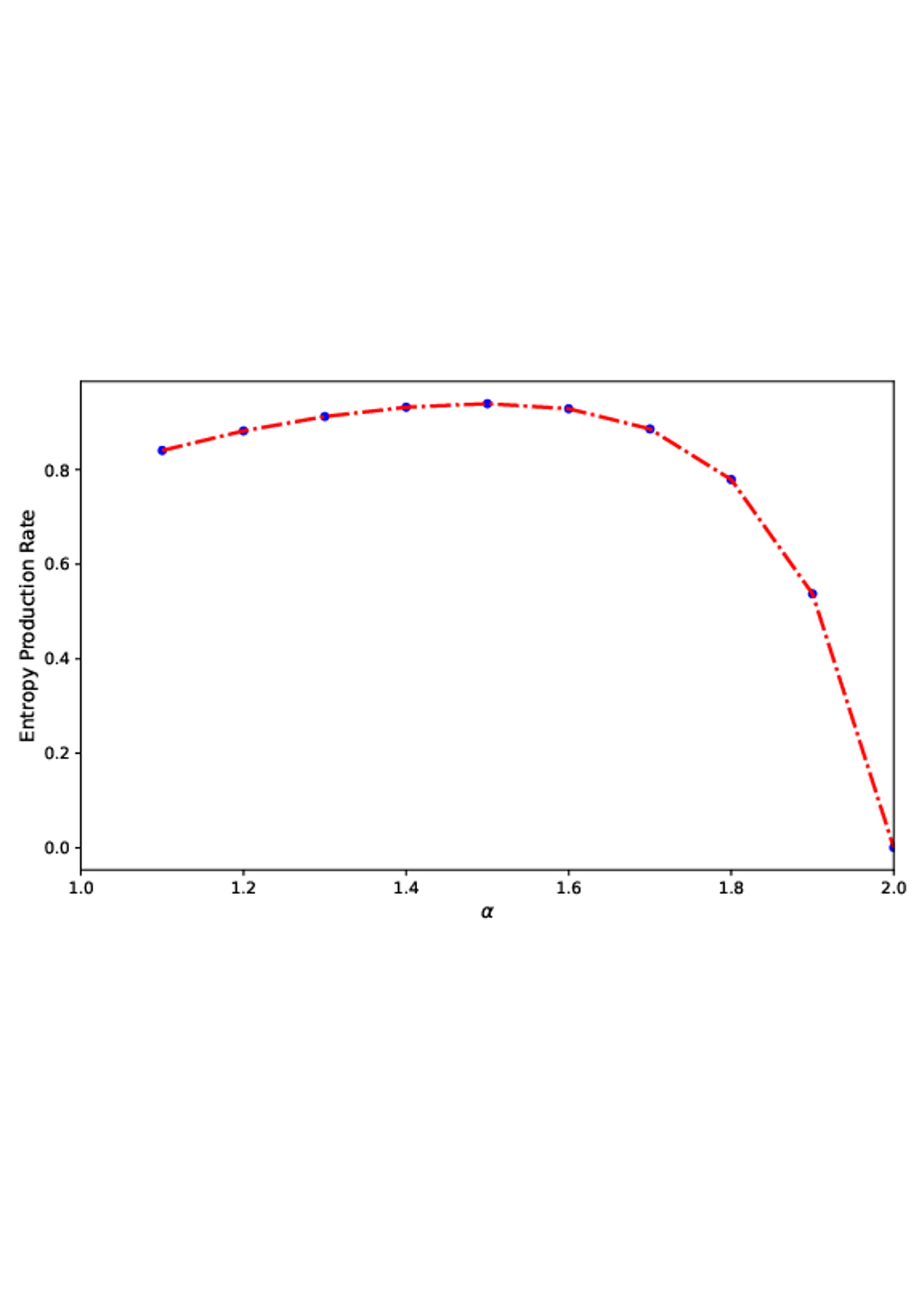}
        \caption{Entropy production rates for different $\alpha$.}
        \label{fig:Ex2_EPR}
    \end{subfigure}
    \caption{The system in Example 2.}
    \label{fig:Ex2}
    \end{figure}


\section{Conclusion}\label{Sec7}

We have developed a thermodynamic framework for general jump-diffusion processes on $\mathbb{R}^n$
Our primary contributions include deriving a mathematical formula for the entropy production rate, and establishing the associated Clausius inequality. We have shown that for a stationary process, this rate is equivalent to the relative entropy between the forward and time-reversed path measures, thereby quantifying the degree of time-irreversibility. Furthermore, we established a key theorem that reveals the equivalence between time-reversibility, zero entropy production, detailed balance, and the system's gradient structure, providing a comprehensive criterion for identifying equilibrium.

Our results are contingent on certain technical assumptions. These include Lipschitz, linear growth, and Lyapunov conditions to ensure the well-posedness of the SDE and the existence of a (stationary) invariant measure. In order to use time reversal and the Girsanov transform for the jump diffusion, we assume that some additional assumptions hold for the jump rate and the density of its invariant measure of the jump diffusion. Although standard, these assumptions may limit the framework's applicability to some physical systems.

This work provides new insights into non-equilibrium thermodynamic and jump diffusions. A natural extension of this work is to consider the associated entropy production fluctuation theorem \cite{MN03} for jump diffusions, which is valid for systems driven arbitrarily far-from-equilibrium. 
Another possible extension is to construct the Onsager reciprocal relations with broken time-reversal symmetry via this entropy production rate formula for jump diffusions. 
Additionally, exploring the link between our entropy production framework and the deep learning for jump diffusion is also exciting direction, following \cite{BV24}.

\section*{Acknowledgments}
The authors thank Prof. Jingqiao Duan, Dr. Qiao Huang, and Dr. Shuyuan Fan, for their fruitful discussions on thermodynamics and stochastic processes.
QZ acknowledge support from the China Postdoctoral Science Foundation (Grant No. 2023M740331).
Particularly, the QZ would like to thank Shu Yang for her generous support and encourage.

\section*{Declarations}
\textbf{Conflict of interest }\quad The authors have no conflicts to disclose.

\section*{Data Availability}
Our manuscript has no associated data.

\end{document}